\documentclass[12pt]{amsart}

\usepackage[backref=page]{hyperref}

\usepackage{fullpage}
\usepackage{hyperref}
\hypersetup{colorlinks=true,citecolor=blue,linkcolor={red!90!black}    }
\usepackage[capitalise]{cleveref}

\usepackage{amssymb, amsmath, amsthm, mathrsfs} 

\usepackage{latexsym}

\usepackage{graphicx}
\date{\today}

\usepackage{tikz,tikz-cd}
\usetikzlibrary{arrows}
\usetikzlibrary{shapes}

\usepackage{enumerate}
\usepackage[capitalise]{cleveref}
\usepackage[normalem]{ulem}

\newtheorem{theorem}{Theorem}[section]
\newtheorem{lemma}[theorem]{Lemma}
\newtheorem{proposition}[theorem]{Proposition}

\theoremstyle{definition}
\newtheorem{definition}[theorem]{Definition}

\theoremstyle{remark}
\newtheorem{remark}[theorem]{Remark}

\numberwithin{equation}{section}

\newcommand{\calO}{\ensuremath{{\mathcal{O}}}}
\newcommand{\calK}{\ensuremath{{\mathcal{K}}}}

\newcommand{\scrO}{\ensuremath{{\mathscr{O}}}}
\newcommand{\scrK}{\ensuremath{{\mathscr{K}}}}
\newcommand{\fraka}{\ensuremath{{\mathfrak{a}}}}
\newcommand{\frakm}{\ensuremath{{\mathfrak{m}}}}

\newcommand{\frakq}{\ensuremath{{\mathfrak{q}}}}

\newcommand{\GLn}{\ensuremath{\mathrm{GL}_n}}

\newcommand{\co}{\ensuremath{\mathcal{O}}}

\newcommand{\m}{\ensuremath{\mathfrak{m}}}

\newcommand{\CalK}{\ensuremath{\mathcal{K}}}

\newcommand{\Qp}{\ensuremath{\overline{\mathbb{Q}}_p}}

\newcommand{\Tr}{\ensuremath{{\mathrm{Tr}}}}
\newcommand{\q}{\ensuremath{{\mathfrak{q}}}}
\newcommand{\frob}{\ensuremath{{\mathrm{Frob}}}}

\title{On some local properties of sequences of big Galois representations}

\author{S. Aniruddha}
\address{Department of Mathematics, Indian Institute of Science Education and Research Bhopal, Bhopal Bypass Road, Bhauri, Bhopal 462066, Madhya Pradesh,
India}
\email{aniruddha18@iiserb.ac.in}
\thanks{}

\author{Jyoti Prakash Saha}
\address{Department of Mathematics, Indian Institute of Science Education and Research Bhopal, Bhopal Bypass Road, Bhauri, Bhopal 462066, Madhya Pradesh,
India}
\curraddr{}
\email{jpsaha@iiserb.ac.in}
\thanks{}

\begin{document}
\subjclass[2000]{11F80} \keywords{Sequences of Galois representations, ramification, potential equivalence, $m$-power character}
\begin{abstract}
In this article, we prove that for a convergent sequence of residually absolutely irreducible representations of the absolute Galois group of a number field $F$ with coefficients in a domain finite over a power series ring over a $p$-adic integer ring, the set of places of $F$ where some of the representations ramifies has density zero. Using this, we extend a result of Das--Rajan to such convergent sequences. We also establish a strong multiplicity one theorem for big Galois representations. 
\end{abstract}
\maketitle 

\section{Introduction}

\subsection{Motivation}
The study of the representations of the absolute Galois groups of number fields is a central theme in arithmetic. Given a smooth projective variety over a number field $F$, its \'etale cohomology groups with coefficients in $\mathbb Q_p$ is a finite dimensional vector space over $\mathbb Q_p$
and 
carries a continuous action of the absolute Galois group of $F$. 
Mazur introduced deformation theory of Galois representations, and under suitable conditions, he proved that the universal deformation has coefficients in complete local Noetherian rings \cite{MazurDeformingGaloisRepr}. 
Hida proved that the $p$-ordinary normalized Hecke eigen cusp forms can be interpolated \cite{HidaGalrepreord, HidaIwasawa}. Hida's work has been extended further by Coleman \cite{ColemanClassicalOverconvergentModularForms}, Coleman and Mazur \cite{ColemanMazurEigencurve}. 
In these works, representations and pseudorepresentations of the absolute Galois groups of number fields with coefficients in rings of large Krull dimension have been studied. 
In this article, we study certain properties of continuous representations of the absolute Galois groups of number fields with coefficients in such rings.

\subsubsection{Places of ramification}

The continuous representations of the absolute Galois groups of number fields, that are of geometric origin, are known to be unramified almost everywhere. Ramakrishna constructed semisimple representations which are ramified at infinitely many places \cite{ramakrishna}. Soon after, Khare and Rajan proved that the set of places of ramification of a continuous semisimple representation of the absolute Galois group of a number field with coefficients in a $p$-adic number field has density zero \cite[Theorem 1]{KhareRajan}. There are related results in different contexts that were obtained by several authors. 
In 2003, Khare showed that for a converging sequence of residually absolutely irreducible $p$-adic representations of the absolute Galois group of a number field $F$ \cite[Definition 1]{Khare-limit}, the set of places of $F$ where at least of one of them ramifies has density zero \cite[Proposition 1]{Khare-limit}. This can be thought of as a refinement of the result of Khare--Rajan. 
The result of Khare--Rajan has also been extended to representations with coefficients in valuation rings of mixed characteristic by Khare--Larsen--Ramakrishna \cite[Theorem 2.5]{KhareLarsenRamakrishnaTranscendentalEllAdic}. 
Further, a related result has also been obtain by  Bella\"{\i}che--Chenevier--Khare--Larsen \cite[Theorem 3.7]{BCKL}.

\subsubsection{Potential equivalence and $m$-power characters}

In a recent work, Das and Rajan proved that for representations of a group with coefficients in a non-archimedean local field, the notions of being potentially equivalent and elementwise potentially equivalent are equivalent \cite[Theorem 2]{das2021finiteness}. Further, for the representations of the absolute Galois groups of number fields, they established
the equivalence between the property of being potentially equivalent, 
having equal $m$-power traces at the Frobenious conjugacy classes at a set of places of density one, and being $m$-trace equivalent \cite[Theorem 4]{das2021finiteness}.

\subsubsection{Big Galois representations}

In 1980's, Hida proved that the ordinary modular forms vary in families. More precisely, for each positive integer $N$ and a prime $p$ with $p\nmid N$ and $Np\geq 4$, he constructed the universal $p$-ordinary Hecke algebra of tame conductor $N$ and showed that the set of its arithmetic specializations are in one-to-one correspondence with the $p$-ordinary $p$-stabilized eigen cusp forms of tame level a divisor of $N$, and
the Galois representations associated with such forms can be interpolated by a big Galois representation through the corresponding arithmetic specializations \cite{HidaGalrepreord, HidaIwasawa}. 
The study of $p$-adic families of automorphic forms have been pursued further by Hida \cite{HidaControThmPNearlyOrdinaryCohGroupSLn}, Coleman \cite{ColemanClassicalOverconvergentModularForms}, Coleman and Mazur \cite{ColemanMazurEigencurve}, Buzzard \cite{BuzzardEigenvarieties}, Chenevier \cite{ChenevierGLn}, Emerton \cite{EmertonEigenVariety}, Urban \cite{UrbanEigenvarieties} et. al. This motivates to look for analogues of the results of Khare--Rajan, Khare et. al. for representations having coefficients in complete local Noetherian rings, for instances, the power series rings over a $p$-adic integer ring. 
It has recently been established by the second author that the result of Khare--Rajan has an analogue in the context of representations with coefficients in rings that are of finite type over power series rings over $p$-adic integer rings \cite[Theorem 1]{saha-den}.

\subsection{Results obtained}
In this article, we show that \cite[Proposition 1]{Khare-limit} also holds for sequences of representations with coefficients in rings which are finite over power series rings over $p$-adic integer rings. More precisely, we establish that if $\calO$ is an integral domain which is finite over a power series ring with coefficients in a $p$-adic integer ring, then for a converging sequence $\{\rho_i\}$ of residually absolutely irreducible representations of the absolute Galois group of a number field $F$ with coefficients in $\calO$, the set of finite places of $F$ where at least of one of the representations $\{\rho_i\}$ ramifies has density zero. 
We refer to \cref{thm:seq-ram}. This result is established by applying \cref{prop:den-of-c-rho-m}, which is an extension of \cite[Proposition 1]{KhareRajan}. 

In Section \ref{Sec:PotenEq}, 
we prove results analogous to some of the results of Das and Rajan \cite{das2021finiteness} for representations with coefficients in rings which are finite over power series rings over $p$-adic integer rings or in rings over affinoid algebras over $p$-adic number fields. 
The proofs of these extensions make use of the results of \cite{das2021finiteness}. We end this section with an extension of \cite[Theorem 4]{das2021finiteness} to a convergent sequence of residually absolutely irreducible Galois representations (\cref{Thm:PotEq=LocPotEq=FromM=MTrace}).

In Section \ref{Sec:StrnogMult1}, we obtain a strong multiplicity one result for big Galois representations (\cref{Thm:StrongMult}). This is an analogue of a result of Rajan \cite[Theorem 1]{RajanStrongMultOne}, who established an analogue of a conjecture of Ramakrishnan in the context of $\ell$-adic Galois representations. We obtain \cref{Thm:StrongMult} as a consequence of \cite[Theorem 1]{RajanStrongMultOne}.

\section{Density of the places of ramification}

In the following, $L$ denotes a $p$-adic number field, and $\calO_L$ denotes its ring of integers. Let $\calO$ be a local domain which is finite over the power series ring $\calO_L[[X_1, \ldots, X_s]]$. Denote the maximal ideal of $\calO$ by $\frakm$, the residue field of $\calO$ by $k$ and the fraction field of $\calO$ by $\calK$. We fix an algebraic closure $\overline{\calK}$ of $\calK$. Let $F$ be a number field with a fixed algebraic closure $\overline F$. Denote  by $G_F$ the absolute Galois group $\mathrm{Gal}(\overline F/F)$ of $F$. Given a place $\q$ of $F$, denote $F_\q$ for the completion of $F$ with respect to the $\q$-adic topology on $F$. 
The decomposition group of $F$ at $\frakq$ is denoted by $G_{F_\frakq}$, which is unique up to conjugates.
Following \cite{Khare-limit}, we introduce the notion of uniform trace convergence for a sequence of representations with coefficients in $\calO$.

 \begin{definition}\label{def:UniTraceConvegence}
     A sequence of Galois representations $\varrho_i:G_F\to \GLn(\calO)$  \emph{uniformly trace converges} to a representation $\varrho:G_F\to \GLn(\calO)$ if the trace of $\varrho_i$ converges to the trace of $\varrho$ uniformly over $G_F$, in the $\m$-adic topology on $\calO$, as $i\to \infty$.
 \end{definition}

\begin{remark}
    Note that \cref{def:UniTraceConvegence} of uniform trace convergence makes sense for representations with coefficients in a local domain $(A,\fraka)$. For another notion related to \textit{uniform trace convergent sequences}, we refer to \cite[Definition 1.1]{BCKL}. 
\end{remark}

In this section, we establish the following result. 
 \begin{theorem}\label{thm:seq-ram}
        Let $\rho_i:G_F\to \GLn(\calO)$ be a sequence of residually absolutely irreducible continuous representations converging to a continuous representation $\rho:G_F\to \GLn(\co)$.  The set of places of $F$ where at least one of the representations $\rho_i$ ramifies has density zero.
 \end{theorem}

The result \cite[Proposition 1]{Khare-limit} of Khare on the places of ramification of a sequence of representations with coefficients in a $p$-adic number field extends, in some sense, the result \cite[Theorem 1]{KhareRajan} of Khare--Rajan on the places of ramification of a single semisimple Galois representation. \cref{thm:seq-ram} extends the result  \cite[Proposition 1]{Khare-limit} 
to the places of ramification of a sequence of Galois representations with coefficients in power series rings over $p$-adic integer rings. Moreover, it also extends the recent result \cite{saha-den} on the places of ramification of a big Galois representation to the places of ramification of a sequence of representations having coefficients in large rings. The key ingredient of the proofs of these results is \cite[Proposition 1]{KhareRajan}.

Before proceeding to the proof of  \cref{thm:seq-ram}, we introduce some notations following \cite[p. 602]{KhareRajan}. Let $(A,\fraka)$ be a local domain and $\Gamma$ be a group.
For a representation $\varrho:\Gamma\to \GLn(A)$ and a positive integer $r$, the reduction of $\varrho$ modulo $\fraka^r$ of $\varrho$ is denoted by $\varrho \bmod{\fraka^r}$.

\begin{definition}

Let $(A,\mathfrak{a})$ be a local domain with residue characteristic $p$, equipped with the $\mathfrak{a}$-adic topology. For a continuous representation  $\varrho:G_F\to \GLn(A)$ and a positive integer $r$, let 
$S_{\varrho,r}$ denote the set of finite places $\frakq$ of $F$ not dividing $p$ and having inertia degree $1$ over $\mathbb Q$ such that the representation $\varrho \bmod \fraka^r$ is unramified at $\frakq$ and the representation $\varrho |_{G_{F_\frakq}} \bmod \fraka^r$ admits a ramified lift over $A$, i.e., there exists a ramified representation $G_{F_\frakq} \to \GLn(A)$ whose reduction modulo $\fraka^r$ is isomorphic to $\varrho|_{G_{F_\frakq}} \bmod \fraka^r$. 
\end{definition}

Now we establish the following proposition, which extends \cite[Proposition 1]{KhareRajan}. Next, applying the proposition below, we will prove \cref{thm:seq-ram}.

\begin{proposition}
    \label{prop:den-of-c-rho-m}
        Let $\rho:G_F\to \GLn(\co)$ be a Galois representation which is semisimple after extending scalars to $ \overline{\calK}$. Then, 
        the upper density of $S_{\rho,m}$ tends to zero, as $m\to \infty$. 
\end{proposition}

To prove the above result, we show that there exists a continuous map $\lambda: \calO \to \calO_K$ where $\calO_K$ denotes the ring of integers of a $p$-adic number field $K$ (cf. \cref{lem:denseinSpecO}) such that $\lambda\circ \rho$ is semisimple and $S_{\rho, r}$ is contained in $S_{\lambda\circ \rho, r}$ for any integer $r\geq 1$, and use \cite[Proposition 1]{KhareRajan}.

\begin{proof}
By \cite[Proposition 4]{saha-den}, there exists an element $h\in \calO$ such that for any ring homomorphism $\lambda: \calO \to \Qp$ with $\lambda(h) \neq 0$, the representation $\lambda\circ \rho$ is semisimple. Given an integer $m\geq 1$ and an element $\frakq\in S_{\rho, m}$, there exists a ramified representation  $\rho_{r, \frakq}: G_{F_\frakq}\to \GLn(\calO)$ which lifts $\rho|_{G_{F_\frakq}} \bmod \frakm^r$. Let $M_{r, \frakq}$ denote a nonidentity element lying in the image of the inertia subgroup of $G_{F_\frakq}$ under $\rho_{r, \frakq}$. 
Consider the set of matrices $M_{r, \frakq}$ as $r$ ranges over the set of positive integers and $\frakq$ ranges over the elements of $S_{\rho, r}$. Note that this set is countable.

Given countably many nonzero elements of $\calO_L[[X_1, \ldots, X_s]]$, their images are nonzero for some continuous $\calO_L$-algebra homomorphism from $\calO_L[[X_1, \ldots, X_s]]\to \Qp$, as can be seen by considering suitable substitutions (for a proof, see \cite[Proposition 3]{saha-den} for instance). 
Since $\calO$ is finite over $\calO_L[[X_1, \ldots, X_s]]$, it follows that given countably many nonzero elements of $\calO$, their images are nonzero for some continuous $\calO_L$-algebra homomorphism from $\calO \to \Qp$. 
It follows that there exists a continuous $\calO_L$-algebra homomorphism $\lambda:\calO\to \overline{\mathbb{Q}}_p$ 
such that $\lambda(h) \neq 0$ and $\lambda (M_{r, \frakq})$ is not equal to the identity matrix over $\Qp$ for any $r\geq 1$ and for any $\frakq$ lying in $S_{\rho, r}$. This implies that $\lambda\circ \rho$ is continuous and semisimple, and the representation $\lambda\circ \rho_{r, \frakq}$ is ramified, for any $r\geq 1$ and $\frakq\in S_{\rho, r}$. 
Since $\calO$ is finite over $\calO_L[[X_1, \ldots, X_s]]$, the image $\lambda(\calO)$ is contained in the valuation ring $\calO_K$ of some $p$-adic number field $K$.
Note that for any $\frakq\in S_{\rho, r}$, the residual representation $\rho \bmod \frakm$ of $\rho$ is unramified at $\frakq$, and hence the residual representation of $\lambda\circ \rho$ is also unramified at $\frakq$. Moreover, the representation $\lambda \circ \rho_{r, \frakq}$ is a lift of the representation $\lambda \circ \rho |_{G_{F_\frakq}} \bmod \frakm_K^r$ where $\frakm_K$ denotes the maximal ideal of $\calO_K$. This shows that the set $S_{\rho,r}$ is contained in $ S_{\lambda\circ \rho,r}$ for any $r\geq 1$. As $\lambda\circ \rho$ is continuous and semisimple, by \cite[Proposition 1]{KhareRajan}, the upper density of $S_{\lambda\circ \rho,r}$ tends to zero, as $r$ tends to infinity. Hence, the result follows.
\end{proof}

\begin{proof}[Proof of \cref{thm:seq-ram}]
Let $\varepsilon > 0$ be given. 
Since $\{\rho_i\}$ is a sequence of residually absolutely irreducible representations and it converges uniformly to $\rho$, it follows from \cite[Th\'eor\`eme 1]{Carayol} that $\rho$ is also  residually absolutely irreducible. In particular, $\rho \otimes \overline{\mathcal K}$ is semisimple. Hence, by \cref{prop:den-of-c-rho-m}, there exists a positive integer $N_\varepsilon$ such that the upper density of $S_{\rho,j}$ is less than $\varepsilon$ for all $j\geq N_\varepsilon$. 
We fix an integer $j\geq N_\varepsilon$. Since $\{\rho_i\}$ converges to $\rho$, there exists a positive integer $N_j$ such that 
the traces of the representations $\rho_i, \rho$ are equal modulo $\frakm^j$ for any $i \geq N_j$. 
So, for any $i \geq N_j$, the residual representations attached to the representations $\rho_i \bmod \frakm^j, \rho \bmod \frakm^j$ have the same traces, and since $\rho_i$ is residually absolutely irreducible, it follows from \cite[Th\'eor\`eme 1]{Carayol} that the representations $\rho_i \bmod \frakm^j, \rho \bmod \frakm^j$ are isomorphic. 

Let $R_i$ denote the set of primes of $F$ where $\rho_i$ is ramified. Since $\rho_i$'s are residually absolutely irreducible, they are absolutely irreducible. Hence, by \cite[Theorem 1]{saha-den}, the set $R_i$ has density zero. This implies that the finite union $\cup_{1\leq i \leq N_j} R_i$ also has density zero.

Let $R'$ be the set of places $\q\in \cup_{i>N_j}R_i$ such that $\frakq$ has inertia degree $>1$ over $\mathbb Q$ or $\rho\bmod{\m^j}$ is ramified at $\frakq$. Since $\calO$ is finite over $\calO_L[[X_1, \ldots, X_s]]$, it follows that the ring $\calO/\frakm^j$ is finite and hence the set of places of ramification of $\rho\bmod{\m^j}$ is finite. 
Therefore, excluding finitely many elements, the set $R'$ is contained in the set of places of $F$ having inertia degree $> 1$. It follows that $R'$ has density zero. 

Let $\q$ be an element of $R_i\setminus R'$ for some $i>N_j$. Note that $\rho_i$ is a  lift of $\rho\bmod{\m^j}$ to $\GLn(\co)$ and $\rho_i$ is ramified at $\q$. Since $\frakq$ does not lie in $R'$, it follows that $\frakq$ has inertia degree $1$ over $\mathbb Q$ and $\rho\bmod{\m^j}$ is unramified at $\frakq$. This shows that $\q$ lies in $S_{\rho,j}$. Consequently, $\cup_{i>N_j}R_i$ is contained in $R'\cup S_{\rho, j}$. It follows that the upper density of $\cup_{i>N_j}R_i$ is less than $\varepsilon$. Consequently, the set  $\cup_{i\geq 1}R_i$ has upper density $< \varepsilon$. This completes the proof. 
\end{proof}

\begin{remark}
It would be interesting to investigate whether \cref{thm:seq-ram} holds for a converging sequence of representations with coefficients in complete local Noetherian rings of characteristic zero and residue characteristic $p$. 
Note that the notion of uniform trace convergence makes sense for a sequence of representations with coefficients in a local ring. By the Cohen structure theorem, a complete local Noetherian ring is the quotient of a power series ring over a $p$-adic integer ring. But we do not know how to establish \cref{thm:seq-ram} when the coefficient ring $\calO$ is a quotient of a such a power series ring. Our method does not seem to provide any insight. This has been one of the obstructions that has been observed in \cite{saha-den} in an approach towards extending \cite[Theorem 1]{KhareRajan} to representations with coefficients in such rings. 
Further, it would also be interesting to investigate whether \cref{prop:den-of-c-rho-m} holds for a representation with coefficients in a complete local Noetherian rings of characteristic zero and residue characteristic $p$. 
\end{remark}

\section{Potential equivalence of representations}
\label{Sec:PotenEq}

Let $L$ be a $p$-adic number field.
Let $\scrO$ be a domain finite over the power series ring $\calO_L[[X_1, \ldots, X_s]]$ or an affinoid algebra over $L$. 
Let $\scrK$ denote the fraction field of $\scrO$, and $\overline{\scrK}$ be a fixed algebraic closure of $\scrK$. 
We recall some definitions and notations used in \cite{das2021finiteness} in the context of representations of a group with coefficients in $\scrO$. Given a representation $\rho:\Gamma\to \GLn(\scrO)$ of a group $\Gamma$ and an integer $m$, the \emph{$m$-trace} of $\rho$ is the map $\chi_\rho^{[m]}:\Gamma\to \scrO$ defined by $g\mapsto \Tr(\rho(g)^m)$.

\begin{definition}Let $\Gamma$ be a group. Two representations $\rho_1,\rho_2:\Gamma\to \GLn(\scrO)$ are said to be  
    \begin{enumerate}[(i)]
        \item  \emph{$m$-trace equivalent} if $\chi_{\rho_1}^{[m]}=\chi_{\rho_2}^{[m]}$,
        \item  \emph{potentially equivalent} if their restrictions to a finite index subgroup of $\Gamma$ are isomorphic over $\scrK$,
        \item  \emph{elementwise potentially equivalent} if, given an element $g\in \Gamma$, there is a positive integer $m_g$ such that $\rho_1(g)^{m_g}$ and $\rho_2(g)^{m_g}$ are conjugates over $\scrK$.
    \end{enumerate}
\end{definition}

We establish the following result which is an analogue of \cite[Theorem 2]{das2021finiteness}.

\begin{theorem}\label{thm:Genpot=elepot=mtrace}
    Let $\rho_1,\rho_2:\Gamma\to \GLn(\scrO)$ be representations which are semisimple after extending scalars to $\CalK$. Then, the following are equivalent.
    \begin{enumerate}[(a)]
        \item $\rho_1,\rho_2$ are potentially equivalent. 
        \item $\rho_1,\rho_2$ are elementwise potentially equivalent.
        \item There is a positive integer $m$, depending only on $\scrO$ and $n$, such that $\rho_1,\rho_2$ are $m$-trace equivalent.
    \end{enumerate}
\end{theorem}
\begin{proof}
It is clear that (a) implies (b). Using \cite[Theorem 1]{das2021finiteness}, the implication from (c)  to (a) follows. To obtain the implication from (b) to (c), 
it is enough to prove the following \cref{lem:weaklemma1}, a weaker version of \cite[Lemma 3]{das2021finiteness} for matrices with coefficients in $\scrO$.
\end{proof}

\begin{lemma}\label{lem:weaklemma1}
There is a positive integer $m$, depending only on $\scrO$ and $n$, with the following property: for $g_1,g_2\in \GLn(\scrO)$, if there is some positive integer $k$ such that $g_1^k$ and $g_2^k$ are conjugates in $\GLn(\scrK)$, then $g_1^m$ and $g_2^m$ have the same traces.  
\end{lemma}
The proof of \cref{lem:weaklemma1} requires the following general lemma. 

\begin{lemma}\label{lem:denseinSpecO}
    There is a finite extension $K/L$ such that the kernels of all the $\co_L$-algebra homomorphisms from $\scrO$ to $K$ form a dense subset of $\mathrm{Spec} (\scrO)$.
\end{lemma}

\begin{proof}
    We prove the result for $\scrO$ finite over a $\co_L[[X_1,\ldots,X_s]]$, the affinoid algebra case is similar as affinoid algebras are finite over Tate algebras (due to Noetherian normalization, \cite[Proposition 3, p. 32]{BoschFormalRigidGeom}). Let $\{x_1,\ldots,x_r\}$ be a generating set of $\scrO$ over $\co_L[[X_1,\ldots,X_s]]$. 
Let $K$ denote the compositum of all subfields of $L'$ of $\Qp$ containing $L$ as a subfield such that the extension $L'/L$ has degree $\leq r$. By Krasner's lemma, $K$ is a finite extension of $L$. 
We see that any element $a\in \scrO$ satisfies a monic polynomial of degree $r$ with coefficients in $\co_L[[X_1,\ldots,X_s]]$. Therefore, if $\lambda:\scrO\to \Qp$ is an $\co_L$-algebra homomorphism, then $\lambda(a)$ satisfies a monic polynomial of degree $r$ over $\co_L$, for all $a\in \scrO$.
This implies that any $\co_L$-algebra homomorphism $\lambda:\scrO\to \Qp$ has image lying inside $\co_K$. Given a non-zero ideal $I$ of $\scrO$ and a non-zero element $a\in I$, by \cite[Proposition 3]{saha-den}, there exists $\lambda:\scrO\to \Qp
$, extending some $\co_L$-algebra homomorphism from $\co_L[[X_1,\ldots,X_s]]\to \co_L$, such that $a\not\in \ker \lambda$. Therefore, given a non-zero ideal $I$ of $\scrO$, there is an $\co_L$-homomorphism $\lambda:\scrO\to K$ such that $I\not\subseteq \ker \lambda$.
\end{proof}

\begin{remark}
By the above lemma, we see that if $a$ is an element of $\scrO$ such that $\lambda(a)=0$ for all $\lambda:\scrO\to K$ as in the hypothesis of \cref{lem:denseinSpecO}, then $a=0$.
\end{remark}

\begin{proof}[Proof of \cref{lem:weaklemma1}]
    From our hypothesis, there exists a matrix $M\in \GLn(\scrK)$ such that 
    $$Mg_1^k=g_2^kM.$$
    By clearing out the denominators (and replacing $M$, if necessary), we can assume that $M$ lies in $M_n(\scrO)$ with $\det M\neq 0$. By \cref{lem:denseinSpecO}, there exists a finite extension $K/L$ such that the kernels of all the $\co_L$-algebra homomorphisms from $\scrO$ to $K$ form a dense subset of $\mathrm{Spec} (\scrO)$. Let $m$ denote the number of roots of unity in the compositum of all degree $n$ extensions of $K$. We prove that the traces of $g_1^m$ and $g_2^m$ are equal.

Let $\lambda$ be an $\co_L$-algebra homomorphism from $\scrO$ to $K$. Suppose 
  $\lambda (\det M)\neq 0$. Then, 
   the matrices $\lambda( g_1)^k$ and $\lambda( g_2)^k$ are conjugates in $\GLn(K)$. By \cite[Lemma 3]{das2021finiteness}, the matrices $\lambda (g_1)^m, \lambda (g_2)^m$ are conjugates over $\GLn(K)$. In particular, we have $\Tr(\lambda (g_1^m))=\Tr(\lambda( g_2^m))$. So, the element $(\Tr(g_1^m)-\Tr(g_2^m))\det M$ lies in the kernels of all $\co_L$-algebra homomorphisms $\lambda:\scrO\to K$, and hence, it is equal to zero by \cref{lem:denseinSpecO}. This shows that the traces   of $g_1^m$ and $g_2^m$ are equal.   
\end{proof}

Let $F$ be a number field and for a non-archimedean place $v$ of $F$, let $F_v$ denote the completion of $F$ with respect to $v$. For an inclusion $i:\overline{F}\hookrightarrow \overline{F}_v$, let $i_v:G_{F_v}\hookrightarrow G_F$ denote the induced inclusion. 
For an unramified place $v$ of $F$, let $\frob_v$ denote the Frobenius element at $v$ in the Galois group $G_{F_v}$. 

\begin{definition}[cf. \cite{das2021finiteness}]
    Let $T$ be a set of places of $F$.
    Two representations $\rho_1,\rho_2:G_F\to \GLn(\scrO)$ are said to be \emph{locally potentially equivalent at $T$}  if  $\rho_1\circ i_v,\rho_2\circ i_v$ are potentially equivalent for all $v\in T$. 
\end{definition}

Note that the above definition is independent of the choice of the inclusion $i:\overline{F}\hookrightarrow \overline{F}_v$.
The following result is an analogue of 
\cite[Theorem 4]{das2021finiteness} which relates potential equivalence with equality of $m$-traces for representations with coefficients in rings of large Krull dimension.

\begin{theorem}
\label{Thm:PotEq=LocPotEq=FromM=MTrace}
Let $F$ be a number field and let $\rho_1,\rho_2:G_F\to \GLn(\scrO)$ be continuous representations which are semisimple after extending scalars to $\overline{\scrK}$.
Then, the following are equivalent:
    
    \begin{enumerate}[(a)]
        \item $\rho_1$ and $\rho_2$ are potentially equivalent.
        \item There exists a set $T$ of places of $F$ with upper density one such that $\rho_1,\rho_2$ are locally potentially equivalent at $T$. 
        \item There exists a positive integer $m$, depending only on $\scrO$ and $n$, and a set $T$ of places of $F$ with upper density one, where $\rho_i$'s are unramified, such that 
        $$\Tr(\rho_1(\frob_v^m))=\Tr(\rho_2(\frob_v^m)) \hspace{5mm}\textnormal{for all }v\in T.$$
        \item The representations $\rho_1,\rho_2$ are $m$-trace equivalent for some positive integer $m$ that depends only on $\scrO$ and $n$.
    \end{enumerate}
\end{theorem}
\begin{proof}
   The implication from (a) to (b) is straightforward, and the implication from  (d) to (a) follows from \cite[Theorem 1]{das2021finiteness}.
    
    Assume that (b) holds. By \cite[Theorem 1]{saha-den}, the set $T_{\mathrm{Ram}}$ of places of $F$ where at least one of the representations $\rho_1, \rho_2$ ramifies has density zero. Hence, the set $T\setminus T_{\mathrm{Ram}}$ has upper density one. Therefore, from \cref{lem:weaklemma1},  $$\Tr(\rho_1(\frob_v^m))=\Tr(\rho_2(\frob_v^m)),$$ where $m$ is a positive integer depending only on $\scrO$ and $n$. This proves that (c) is true.  
    
     Assume that (c) holds. By \cite[Proposition 4]{saha-den}, there is an element $h\neq 0$ in $\scrO$ such that $\lambda \circ \rho_1, \lambda\circ \rho_2$ are semisimple for any $\lambda: \scrO \to \Qp$ with $\lambda(h) \neq 0$. 
     By \cref{lem:denseinSpecO}, there exists a finite extension $K/L$ such that the kernels of all the $\co_L$-algebra homomorphisms from $\scrO$ to $K$ form a dense subset of $\mathrm{Spec} (\scrO)$. Let $\lambda: \scrO \to K$ be a continuous $\calO_L$-algebra homomorphism. Suppose $\lambda(h) \neq 0$. Note that $\lambda\circ \rho_1, \lambda \circ \rho_2$ are unramified at $T$. By \cite[Theorem 4]{das2021finiteness}, the representations $\lambda\circ \rho_1, \lambda \circ \rho_2$ are $m$-trace equivalent. Then for any $g\in G_F$, $\lambda (\Tr \rho_1(g^m) - \Tr \rho_2(g^m))=0$, and hence $(\Tr \rho_1(g^m) - \Tr \rho_2(g^m))h$ vanishes under any continuous $\calO_L$-algebra homomorphism $\lambda: \scrO \to K$. Hence,  \cref{lem:denseinSpecO}, it follows that $(\Tr \rho_1(g^m) - \Tr \rho_2(g^m))h=0$. This proves the implication from (c) to (d).
\end{proof}

Note that the above result admits an immediate extension to a finite number of representations. We establish the following result that considers a converging sequences of  residually absolutely irreducible Galois representations with coefficients in $\co$, a domain finite over a power series ring of a $p$-adic integer ring.

\begin{theorem}
 Let $F$ be a number field and let $\rho_i:G_F\to \GLn(\calO)$ be sequences of residually absolutely irreducible Galois representations uniformly trace converging to a representation $\rho:G_F\to \GLn(\calO)$. 
Then, the following are equivalent:
    
    \begin{enumerate}[(a)]
        \item For any $i,j\geq 1$, the representations $\rho_i$ and $\rho_j$ are potentially equivalent.
        \item There exists a set $T$ of places of $F$ with upper density one such that, for all $i,j\geq 1$, the representations $\rho_i$ and $\rho_j$ are locally potentially equivalent at $T$. 
        \item There exists a positive integer $m$, depending only on $\calO$ and $n$, and a set $T$ of places of $F$ with upper density one, where $\rho_i$'s are unramified, such that 
        $$\Tr(\rho_i(\frob_v^m))=\Tr(\rho_j(\frob_v^m)) \hspace{5mm}\textnormal{for all }v\in T\textnormal{ and }i,j\geq 1.$$
        \item For all $i,j\geq 1$, the representations $\rho_i,\rho_j$ are $m$-trace equivalent for some $m$.
    \end{enumerate}
\end{theorem}
\begin{proof}
The implication from (a) to (b) is immediate.
When (b) holds, from 
\cref{thm:seq-ram}, it follows that there is a set $T$ of places of $F$ with upper density one where $\rho_i$'s are unramified, and then we obtain a positive integer $m$ having the required properties by applying \cref{lem:weaklemma1}. This shows that (b) implies (c). 
\cref{Thm:PotEq=LocPotEq=FromM=MTrace} yields the remaining implications. 
\end{proof}

 Let $(A,\fraka)$ be a local domain. Note that the notion of $m$-trace convergence can be defined for sequences of representations of a group with coefficients in a local domain. 
 \begin{definition}
 Let $(A,\fraka)$ be a local domain. 
     A sequence of representations $\varrho_i:\Gamma\to \GLn(A)$ is said to be  \emph{uniformly $m$-trace convergent} 
     if  $\Tr(\varrho_i(g)^m)$ is convergent, uniformly over $\Gamma$, in the $\fraka$-adic topology on $A$, as $i\to \infty$.
 \end{definition}
 
 From \cite[Theorem 1]{das2021finiteness}, a sequence $\rho_i:\Gamma\to \GLn(A)$ where, for all $i,j\geq 1$, the representations  $\rho_i,\rho_j$ are potentially equivalent, is $m$-trace convergent. It would be interesting to see if \cite[Proposition 1]{Khare-limit} can be extended to uniformly $m$-trace convergent sequences of residually absolutely irreducible Galois representations $G_F\to \GLn(A)$, where $A$ is either the valuation ring of a $p$-adic number field, or a domain finite over a power series ring over a $p$-adic integer ring.

\section{Strong multiplicity one theorems}
\label{Sec:StrnogMult1}

The multiplicity one theorems are fundamental in the study of automorphic representations and Galois representations. 
Let $F$ be a number field and $\mathbb A_F$ denote the ring of adeles of $F$. 
For automorphic representations of $\mathrm{GL}_n(\mathbb A_F)$, the strong multiplicity one theorem due to
Jacquet, Piatetskii--Shapiro and Shalika \cite{JacquetShalikaI, JacquetShalikaII, JacquetPiatetskiiShapiroShalika} states that two irreducible, unitary, cuspidal automorphic representations are isomorphic if their local factors are isomorphic at almost all finite places. 
Ramakrishnan conjectured a refinement of this theorem for automorphic representations  \cite[p. 442]{RamakrishnanPureMotivesAutomorphicForms}, which he proved for $n=2$ \cite{RamakrishnanRefinementStrongMult}. 
The analogue of the conjecture due to Ramakrishnan for $\ell$-adic Galois representations was established by Rajan \cite[Theorem 1]{RajanStrongMultOne}. 
We show that the arguments of the previous sections also yield a strong multiplicity one result for big Galois representations. This is obtained as a consequence of \cite[Theorem 1]{RajanStrongMultOne}.

\begin{theorem}
\label{Thm:StrongMult}
Let $\scrO$ be a domain finite over $\co_L[[X_1,\ldots,X_s]]$, or be an affinoid algebra over $L$. 
Let $\varrho_1,\varrho_2:G_F\to \GLn(\scrO)$ be two continuous representations such that they are semisimple after extending scalars to $\overline \scrK$. 
Suppose they are unramified outside a finite set $S$ of finite places of $F$, and the upper density of 
$$SM(\varrho_1,\varrho_2)=\{v\in \Sigma_F \setminus S\ |\ \Tr(\rho_1(\frob_v))=\Tr(\rho_2(\frob_v))\}$$ 
is $>1-1/2n^2$. Then $\rho_1, \rho_2$ are isomorphic after extending scalars to the fraction field of $\scrO$.
\end{theorem}

\begin{proof}
By \cite[Proposition 4]{saha-den}, there exists a nonzero element $h\in \scrO$ such that $\lambda \circ \varrho_1, \lambda \circ \varrho_2$ are semisimple for any $\calO_L$-algebra homomorphism $\lambda: \scrO \to \Qp$ with $\lambda (h) \neq 0$. 
Since the upper density of $SM(\varrho_1,\varrho_2)$ is $> 1 - \frac 1{2n^2}$, for such a homomorphism $\lambda$, we obtain from \cite[Theorem 1]{RajanStrongMultOne} that $\lambda \circ \varrho_1, \lambda \circ \varrho_2$ are isomorphic, and hence $\lambda \circ \Tr \varrho_1, \lambda \circ \Tr \varrho_2$ are equal. 
By \cref{lem:denseinSpecO}, the kernels of the 
$\calO_L$-algebra homomorphisms $\lambda: \scrO \to \Qp$ with $\lambda (h) \neq 0$ form a dense subset of $\mathrm{Spec}(\scrO)$. It follows that $\varrho_1, \varrho_2$ have the same traces. By the Brauer--Nesbitt theorem, the result follows. 
\end{proof}
 
\section*{Acknowledgement}
 
The second author would like to acknowledge the INSPIRE Faculty Award (IFA18-MA123) from the Department of Science and Technology, Government of India. 

\providecommand{\bysame}{\leavevmode\hbox to3em{\hrulefill}\thinspace}
\providecommand{\MR}{\relax\ifhmode\unskip\space\fi MR }
\providecommand{\MRhref}[2]{%
  \href{http://www.ams.org/mathscinet-getitem?mr=#1}{#2}
}
\providecommand{\href}[2]{#2}


\begin{thebibliography}{BCKL05}

\bibitem[BCKL05]{BCKL}
Jo\"{e}l Bella\"{\i}che, Ga\"{e}tan Chenevier, Chandrashekhar Khare, and
  Michael Larsen, \emph{Converging sequences of {$p$}-adic {G}alois
  representations and density theorems}, Int. Math. Res. Not. (2005), no.~59,
  3691--3720. \MR{2200082}

\bibitem[Bos14]{BoschFormalRigidGeom}
Siegfried Bosch, \emph{Lectures on formal and rigid geometry}, Lecture Notes in
  Mathematics, vol. 2105, Springer, Cham, 2014. \MR{3309387}

\bibitem[Buz07]{BuzzardEigenvarieties}
Kevin Buzzard, \emph{Eigenvarieties}, {$L$}-functions and {G}alois
  representations, London Math. Soc. Lecture Note Ser., vol. 320, Cambridge
  Univ. Press, Cambridge, 2007, pp.~59--120. \MR{2392353}

\bibitem[Car94]{Carayol}
Henri Carayol, \emph{Formes modulaires et repr\'{e}sentations galoisiennes \`a
  valeurs dans un anneau local complet}, {$p$}-adic monodromy and the {B}irch
  and {S}winnerton-{D}yer conjecture ({B}oston, {MA}, 1991), Contemp. Math.,
  vol. 165, Amer. Math. Soc., Providence, RI, 1994, pp.~213--237. \MR{1279611}

\bibitem[Che04]{ChenevierGLn}
Ga{\"e}tan Chenevier, \emph{Familles {$p$}-adiques de formes automorphes pour
  {${\rm GL}_n$}}, J. Reine Angew. Math. \textbf{570} (2004), 143--217.
  \MR{2075765}

\bibitem[CM98]{ColemanMazurEigencurve}
R.~Coleman and B.~Mazur, \emph{The eigencurve}, Galois representations in
  arithmetic algebraic geometry ({D}urham, 1996), London Math. Soc. Lecture
  Note Ser., vol. 254, Cambridge Univ. Press, Cambridge, 1998, pp.~1--113.
  \MR{1696469}

\bibitem[Col96]{ColemanClassicalOverconvergentModularForms}
Robert~F. Coleman, \emph{Classical and overconvergent modular forms}, Invent.
  Math. \textbf{124} (1996), no.~1-3, 215--241. \MR{1369416}

\bibitem[DR21]{das2021finiteness}
Plawan Das and C.~S. Rajan, \emph{Finiteness theorems for potentially
  equivalent {G}alois representations: extension of {F}altings' finiteness
  criteria}, 2021.

\bibitem[Eme06]{EmertonEigenVariety}
Matthew Emerton, \emph{On the interpolation of systems of eigenvalues attached
  to automorphic {H}ecke eigenforms}, Invent. Math. \textbf{164} (2006), no.~1,
  1--84. \MR{2207783}

\bibitem[Hid86a]{HidaGalrepreord}
Haruzo Hida, \emph{Galois representations into {${\rm GL}_2({\bf Z}_p[[X]])$}
  attached to ordinary cusp forms}, Invent. Math. \textbf{85} (1986), no.~3,
  545--613. \MR{848685}

\bibitem[Hid86b]{HidaIwasawa}
\bysame, \emph{Iwasawa modules attached to congruences of cusp forms}, Ann.
  Sci. \'Ecole Norm. Sup. (4) \textbf{19} (1986), no.~2, 231--273. \MR{868300}

\bibitem[Hid95]{HidaControThmPNearlyOrdinaryCohGroupSLn}
\bysame, \emph{Control theorems of {$p$}-nearly ordinary cohomology groups for
  {${\rm SL}(n)$}}, Bull. Soc. Math. France \textbf{123} (1995), no.~3,
  425--475. \MR{1373742}

\bibitem[JPSS83]{JacquetPiatetskiiShapiroShalika}
H.~Jacquet, I.~I. Piatetskii-Shapiro, and J.~A. Shalika, \emph{Rankin-{S}elberg
  convolutions}, Amer. J. Math. \textbf{105} (1983), no.~2, 367--464.
  \MR{701565}

\bibitem[JS81a]{JacquetShalikaII}
H.~Jacquet and J.~A. Shalika, \emph{On {E}uler products and the classification
  of automorphic forms. {II}}, Amer. J. Math. \textbf{103} (1981), no.~4,
  777--815. \MR{623137}

\bibitem[JS81b]{JacquetShalikaI}
\bysame, \emph{On {E}uler products and the classification of automorphic
  representations. {I}}, Amer. J. Math. \textbf{103} (1981), no.~3, 499--558.
  \MR{618323}

\bibitem[Kha03]{Khare-limit}
Chandrashekhar Khare, \emph{Limits of residually irreducible {$p$}-adic
  {G}alois representations}, Proc. Amer. Math. Soc. \textbf{131} (2003), no.~7,
  1999--2006. \MR{1963742}

\bibitem[KLR05]{KhareLarsenRamakrishnaTranscendentalEllAdic}
Chandrashekhar Khare, Michael Larsen, and Ravi Ramakrishna,
  \emph{Transcendental {$l$}-adic {G}alois representations}, Math. Res. Lett.
  \textbf{12} (2005), no.~5-6, 685--699. \MR{2189230}

\bibitem[KR01]{KhareRajan}
Chandrashekhar Khare and C.~S. Rajan, \emph{The density of ramified primes in
  semisimple {$p$}-adic {G}alois representations}, Internat. Math. Res. Notices
  (2001), no.~12, 601--607. \MR{1836789}

\bibitem[Maz89]{MazurDeformingGaloisRepr}
B.~Mazur, \emph{Deforming {G}alois representations}, Galois groups over {${\bf
  Q}$} ({B}erkeley, {CA}, 1987), Math. Sci. Res. Inst. Publ., vol.~16,
  Springer, New York, 1989, pp.~385--437. \MR{1012172}

\bibitem[Raj98]{RajanStrongMultOne}
C.~S. Rajan, \emph{On strong multiplicity one for {$l$}-adic representations},
  Internat. Math. Res. Notices (1998), no.~3, 161--172. \MR{1606395}

\bibitem[Ram94a]{RamakrishnanPureMotivesAutomorphicForms}
Dinakar Ramakrishnan, \emph{Pure motives and automorphic forms}, Motives
  ({S}eattle, {WA}, 1991), Proc. Sympos. Pure Math., vol.~55, Amer. Math. Soc.,
  Providence, RI, 1994, pp.~411--446. \MR{1265561}

\bibitem[Ram94b]{RamakrishnanRefinementStrongMult}
\bysame, \emph{A refinement of the strong multiplicity one theorem for {${\rm
  GL}(2)$}. {A}ppendix to: ``{$l$}-adic representations associated to modular
  forms over imaginary quadratic fields. {II}'' [{I}nvent. {M}ath. {\bf 116}
  (1994), no. 1-3, 619--643; {MR}1253207 (95h:11050a)] by {R}. {T}aylor},
  Invent. Math. \textbf{116} (1994), no.~1-3, 645--649. \MR{1253208}

\bibitem[Ram00]{ramakrishna}
Ravi Ramakrishna, \emph{Infinitely ramified {G}alois representations}, Ann. of
  Math. (2) \textbf{151} (2000), no.~2, 793--815. \MR{1765710}

\bibitem[Sah19]{saha-den}
Jyoti~Prakash Saha, \emph{The density of ramified primes}, Doc. Math.
  \textbf{24} (2019), 2423--2429. \MR{4044364}

\bibitem[Urb11]{UrbanEigenvarieties}
Eric Urban, \emph{Eigenvarieties for reductive groups}, Ann. of Math. (2)
  \textbf{174} (2011), no.~3, 1685--1784. \MR{2846490}

\end{thebibliography}
 \end{document}